\documentclass[10pt]{amsart}
\usepackage{amssymb}
\usepackage{epsfig}
\usepackage{url}
\usepackage{setspace}
\usepackage{pdflscape}
\theoremstyle{plain}

\newtheorem{thm}{Theorem}[section]
\newtheorem{cor}[thm]{Corollary}

\newtheorem{prop}[thm]{Proposition}
\newtheorem{rem}[thm]{Remark}
\newtheorem{ques}[thm]{Question}
\newtheorem{conj}[thm]{Conjecture}
\newtheorem{exam}[thm]{Example}
\def\cal{\mathcal}
\def\bbb{\mathbb}
\def\op{\operatorname}
\renewcommand{\phi}{\varphi}
\newcommand{\R}{\bbb{R}}
\newcommand{\N}{\bbb{N}}
\newcommand{\Z}{\bbb{Z}}
\newcommand{\Q}{\bbb{Q}}

\begin{document}

\title[On the Diophantine equation $x^3\pm y^3=a^k\pm b^k$]{On the Diophantine equation $x^3\pm y^3=a^k\pm b^k$}
\author{Maciej Ulas}


\begin{abstract} In this note we consider the title Diophantine equation from both a theoretical as well as experimental point of view. In particular, we prove that for $k=4, 6$ and each choice of the signs our equation has infinitely many coprime positive integer solutions $(x, y, a, b)$ such that no partial sum in the expression $x^3 \pm y^3-(a^k \pm b^k)$ vanishes. The same is true for each $k\not\equiv 0\pmod{4}$ and the equation $x^3\pm y^3=a^k-b^k$. For $k=5, 7$ and all choices of the signs we computed all coprime positive integer solutions $(x, y, a, b)$ of $x^3\pm y^3=a^k+b^k$ satisfying the condition $b<a\leq 50000$.
\end{abstract}

\maketitle

\section{Introduction}\label{sec1}
The problem of finding integer solutions to diagonal equations of the form  
$$
\sum_{i=1}^{s}a_{i}x_{i}^m=0,
$$
where $a_{i} \in \mathbb{Z}$ are fixed, coprime, and nonzero integers, is a classical one. Due to the homogeneity of the underlying polynomial, there is no fundamental distinction between rational and integer solutions.  

In the case where $s = m = 3$, the well-developed theory of elliptic curves provides a number of methods to study the existence of solutions. When \( s = 3 \) and \( m > 3 \), the problem leads to the study of Fermat-type curves, including the famous Fermat conjecture, which asserts the nonexistence of nonzero integer solutions to the equation \( x^n + y^n = z^n \) for \( n > 2 \), as proved by Wiles.  

For $s = 4$, we move into the realm of algebraic surfaces. Cubic surfaces ($m =3$) and quartic surfaces ($m = 4$) have been extensively studied, leading to significant results. A striking example is the equation $x_1^4 + x_2^4 + x_3^4 = x_4^4$, for which it has been shown by Elkies \cite{Elk} that the set of rational points is dense (in the Euclidean topology) within the set of real points.  

For $s > 3$ and $m > 4$, the picture becomes far less clear, and many fundamental questions remain open. We believe that the most famous conjecture in this setting is Euler's conjecture, which predicts the nonexistence of nontrivial integer solutions of the equation $x_{1}^m+\cdots+x_{m-1}^{m}=x_{m}^m$ for $m \geq 6$. As we mentioned above, Elkies showed that this conjecture is not true for $m=4$ and Lander and Parkin \cite{LP} found a counterexample for $m=5$, demonstrating that the situation is more intricate than initially believed. However, we don't know whether there are infinitely many nontrivial solutions in the case $m=5$ and whether there is {\it any} nontrivial solution for $m\geq 6$.  

Our understanding of the solvability of diagonal equations with varying exponents, i.e., equations of the form  
$$
\sum_{i=1}^{s}a_{i}x_{i}^{m_i}=0,
$$
where $ m_1, \ldots, m_s \in \mathbb{N}$ are positive integers (with at least two $i, j$ such that $m_{i}\neq m_{j}$), is even more limited. Except for the case $s = 3$ with certain specific exponents—where one can work within an appropriate weighted projective space—there is no developed theory for the existence of coprime integer solutions.  

In this note, for a fixed integer $k \in \N_{\geq 4}$, we consider the Diophantine equation
\begin{equation}\label{generaleq}
x^3 \pm y^3 = a^k \pm b^k,
\end{equation}
where the choices of the ``$\pm$'' signs are made independently. Our main interest is the existence of solutions in coprime positive integers. It is known that in the special case $x^3 + y^3 = a^k + b^k$, there are infinitely many integer solutions with $\{x^3, y^3\} \neq \{a^k, b^k\}$. This result follows from a theorem of Lander \cite{Lan}, which can be extended to the mixed-sign equations as well, although Lander’s construction does not produce coprime solutions.

When $k > 3$, finding coprime solutions becomes more challenging. The reason is that the polynomial $x^3 \pm y^3 - (a^k \pm b^k) \in \Z[x,y,a,b]$ is not homogeneous, so it is not obvious how to eliminate common divisors of $x, y, a, b$. Of course, one can do this in specific cases. For instance, if there exists $d \in \N_{\geq 2}$ such that $d^k \mid x, y$ and $d^3 \mid a, b$, then dividing through by $d^{3k}$ yields a solution in coprime integers, provided $\gcd(a,b,x,y) = d^m$ with $m = \min\{3, k\}$. In general, however, such a divisibility condition may fail.

Recently, Wagstaff \cite{Wag} studied the special case of \eqref{generaleq} with both signs set to ``$+$'' from a computational standpoint. In particular, he found all coprime positive solutions in the range $a \leq b \leq M$, where $M = 10000, 5000, 1400, 700$ for $k=4, 5, 6, 7$, respectively. Moreover, he gave heuristic arguments suggesting that if $k = 4, 5, 6$, then the equation $x^3 + y^3 = a^k + b^k$ has infinitely many solutions in coprime integers, whereas for $k \ge 7$, there are only finitely many such solutions.

Motivated by Wagstaff’s results, we investigate the more general equation \eqref{generaleq} both computationally and theoretically. We say that a quadruple $(x, y, a, b)$ of integers is a \emph{nontrivial} solution of \eqref{generaleq} if $x, y, a, b$ are coprime, and no partial sum in the expression
\[
x^3 \pm y^3 \;-\; (a^k \pm b^k)
\]
vanishes. In this context, we also introduce the counting function
\begin{align*}
C_{k}(\pm, \pm)(N)
=
\#&\{(x, y, a, b)\in\N_{+}^4:\; (x, y, a, b) \text{ is a nontrivial solution of } \\
  &\quad\quad 0 < x^3 \pm y^3 = a^k \pm b^k \leq N\},
\end{align*}
and in what follows, we establish lower bounds for $C_{k}(\pm, \pm)(N)$ for several values of $k$ and suitable choices of signs.

Let us describe the content of the paper in some detail. In Section \ref{sec2}, we prove that for $k=4$ and each choice of the signs the Diophantine equation (\ref{generaleq}) has infinitely many solutions in coprime positive integers. An analogous result is proved for the case of $k=6$ in Section \ref{sec3}. Numerical results of extensive computations concerning the Diophantine equation (\ref{generaleq}) for $k=5, 7$ are presented in Section \ref{sec4} as well as general result concerning the infinitude of nontrivial solutions of the equation $x^3\pm y^3=a^k-b^k$, where $k\not\equiv 0\pmod{4}$. Finally, in the last section we present some additional remarks and questions concerning the equation $x^3\pm y^3=a^k+b^k$ for $k=4, 6$.  

\section{The equation $x^3\pm y^3=a^4\pm b^4$ in positive integers}\label{sec2}

In this section we investigate the equation (\ref{generaleq}) for $k=4$. We start with the simple result.

\begin{prop}\label{minussing}
For each choice of the sign the Diophantine equation $x^3\pm y^3=a^4-b^4$ has infinitely many nontrivial solutions in positive integers.
\end{prop}
\begin{proof}
In case of the ``$+$'' sign we take coprime positive integers $u, v$ of different parity and note the identity
\begin{equation}\label{idenpm4}
(2u^3v)^3+(2uv^3)^3=(u^3+v^3)^4-(u^3-v^3)^4.
\end{equation}
In particular, for a given $N\in\N_{+}$, the inequality $C_{4}(+,-)(N)\gg N^{1/9}$ is true.

To get the solutions in the case of ``$-$'' we performed a small numerical search of the positive integer solutions of the equation $x^3-y^3=a^4-b^4$ in the range $0<b<a<10^3$. In this range there are 197 coprime solutions. A quick analysis of the set of solutions revealed that the $a$ coordinate very often satisfies the condition $a=3(2t+1)^2$ for some $t\in\N$. Using this property we quickly found that the corresponding value of $b$ is $b=(3t+1)(3t+2)$. Having conjectured form of $a, b$ in terms of $t$, we quickly found the corresponding expressions for $x, y$. More precisely, we have the following identity
$$
(15t^3+33 t^2+21t+4)^3-(15t^3+12t^2-1)^3=(3(2t+1)^2)^4-((3t+1)(3t+2))^4.
$$
Because for each $t\in\N_{+}$ we have $\gcd(3(2t+1),(3t+1)(3t+2))=1$ we get infinitely many nontrivial solutions in positive integers of our equation. In consequence, the inequality $C_{4}(-,-)(N)\gg N^{1/8}$ is true.
\end{proof}

In the proof of the above proposition we observed that $C_{4}(+,-)(N)\gg N^{1/9}$. Using the parametric solution of the equation $x^3+y^3=a^4-b^4$ we can obtain the better result.  
\begin{cor}
For $N\gg 1$ we have $C_{4}(+,-)(N)\gg \frac{3}{\pi^2 2^{5/3}}N^{1/6}$.
\end{cor}
\begin{proof}
Let us back to the identity (\ref{idenpm4}) with the condition that $u, v$ has different parity, and note the following chain of inequalities
\begin{align*}
C_{4}(+,-)(N)&\geq \frac{1}{2}\sharp \{(u, v)\in \N_{+}^2:\; u\leq v\wedge \gcd(u, v)=1 \wedge 8u^3v^3(u^6+v^6)\leq N\}\\
             &\geq \frac{1}{2}\sharp \{(u, v)\in \N_{+}^2:\; u\leq v\wedge \gcd(u, v)=1 \wedge 16u^3v^9\leq N\}\\
             &\geq \frac{1}{2}\sharp \{(u, v)\in \N_{+}^2:\; u\leq v\wedge \gcd(u, v)=1 \wedge 16v^{12}\leq N\}\\
             &=\frac{1}{2}\sum_{v\leq (N/16)^{1/12}}\sum_{u\leq v, \gcd(u, v)=1}1=\frac{1}{2}\sum_{v\leq (N/16)^{1/12}}\phi(v)
\end{align*}
Invoking now the classical summation formula $\sum_{n\leq x}\phi(n)=\frac{3}{\pi^2}x^2+O(x\log x)$, and apply it with $x=(N/16)^{1/12}$ we get the statement. 
\end{proof}

\begin{rem}
{\rm It would be nice to get more precise lower bound of $C_{4}(-,-)(N)$ then the one presented in the proof of Theorem \ref{minussing}. By replacing $t$ by $u/v^4$ with $\gcd(u,v)=1$ heuristic reasoning suggest that we can expect $C_{4}(-,-)(N)\gg N^{1/8+1/36}$. }
\end{rem}

The equation $x^3\pm y^3=a^4+b^4$ is more difficult. At the beginning of our investigations we were trying to find polynomial solutions of our equation.  However, we were unable to do so and decided to investigate these equations using a different approach. More precisely, instead of working with the equation $x^3+y^3=a^4+b^4$, we were working with the simpler equation $X^3+Y^3=A^4+B^2$ and looked for a solution in (not necessarily coprime) polynomials. In fact, we assumed the existence of polynomials $X, Y, A, B\in\Z[t]$ with $\op{deg}X=\op{deg}Y=3, \op{deg}A=2, \op{deg}B\leq 3$, such that $X(t)^3+Y(t)^3=A(t)^4+tB(t)^2$ and then replace $t$ by $x^2$. To simplify computations we put
$$
X(t)=t(t^2-(a + b)t + c), \;Y(t)=-t(t^2-(a-b)t+ c)
$$
and $A(t)=t(pt + q), B(t)=t^2(rt+s)$, where $a, b, c, p, q, r, s$ need to be determined. Under these assumptions we have $X(t)^3+Y(t)^3-A(t)^4-tB(t)^2=\sum_{i=0}^{4}C_{i}t^{i+4}$, and consider the system of equations $C_{i}=0$ for $i=0, 1, 2, 3, 4$. It can be easily solved with respect to $a, b, c, q, s$ and we get
$$
a=-\frac{4 p^3 q+r^2}{2 p^4},\quad b=-\frac{p^4}{6},\quad c=\frac{q^2}{p^2},\quad q=-\frac{p^{16}+27 r^4}{432 p^3 r^2},\quad s=\frac{p^{16}+27 r^4}{432 p^4 r}.
$$
To find the simplest possible expressions for the polynomials we are looking for, we take $p=6d, r=144 d^2$, where $d$ is a variable, and replace $t$ by $x^2$. After all necessary simplifications we obtain the identity
\begin{equation}\label{34identity}
f_{1}(x)^3+f_{2}(x)^3=g_{1}(x)^4+g_{2}(x)^2,
\end{equation}
where
\begin{align*}
f_{1}(x)&=x^2(x^4-6(81 d^8-36d^4-1)x^2+(243 d^8+1)^2),\\
f_{2}(x)&=-x^2(x^4-6(81 d^8+36d^4-1)x^2+(243 d^8+1)^2),\\
g_{1}(x)&=6dx^2(x^2-243d^8-1),\\
g_{2}(x)&=(12dx^2)^2x(x^2+243d^8+1).
\end{align*}

We are ready to prove the following.

\begin{thm}\label{equation34}
For each choice of the sign the Diophantine equation $x^3\pm y^3=a^4+b^4$ has infinitely many nontrivial solutions in positive integers.
\end{thm}
\begin{proof}
We consider the identity (\ref{34identity}) with $d=2$. We define the sets
\begin{align*}
\cal{A}_{1}&=\{x\in\R\setminus\{0\}:\;f_{1}(x)>0\;\mbox{and}\;f_{2}(x)>0\}=(-r_{2}, -r_{1})\cup (r_{1}, r_{2}),\\
\cal{A}_{2}&=\{x\in\R\setminus\{0\}:\;f_{1}(x)>0\;\mbox{and}\;f_{2}(x)<0\}\\
           &= (-\infty, -r_{2})\cup (-r_{1},0)\cup (0, r_{1})\cup (r_{2}, +\infty),
\end{align*}
 where 
 $$
 r_{1}=\sqrt{63933-2 \sqrt{54367202}}\approx221.779, \quad r_{2}=\sqrt{63933+2 \sqrt{54367202}}\approx 280.499.
 $$
Observe that $f_{2}(x)=0$ if and only if $x=0$ or $x\in\{\pm r_{1}, \pm r_{2}\}$.
In particular, the sets $\cal{A}_{1}, \cal{A}_{2}$ are infinite and each contains infinitely many rational numbers. After these initial observations we note that if for some $x_{0}\in\Q$ the value $g_{2}(x_{0})$ is a square, then there exists $y_{0}\in\Q$ such that the point $(x_{0}, y_{0})$ lies on the elliptic curve
$$
E:\;y^2=x^3+62209x.
$$
Let $(x, y)$ be a rational point on $E$. Then, there are $u, v\in\Z, w\in\N$ satisfying the condition $\gcd(uv,w)=1$ and $x=u/w^2, y=v/w^3$. Equivalently the integers $u, v, w$ satisfy the identity $v^2=u^3+62209uw^4$. Thus, by substituting $x=u/w^2$ and multiplying both sides by $w^{36}$, we can rewrite the identity (\ref{34identity}) in the following way
\begin{align*}
(3464u^4w^4+(v^2-2u^3)^2)^3&+(3448u^4w^4-(v^2-2u^3)^2)^3\\
                                             &=(12uw(2u^3-v^2))^4+(24u^2vw^2)^4.
\end{align*}
To finish the proof we need to show how from the above identity we can construct coprime solutions. Let $g=\gcd(u, v)$. We write $u=pg, v=qg, w=r$, where $\gcd(p,q)=1$ and $\gcd(pq,r)=1$ and $gq^2=g^2p^3+62209pr^4$. After dividing by $g^{12}$ we get that
$h_{1}^3+h_{2}^3=h_{3}^4+h_{4}^4$, where
\begin{align*}
h_{1}&=3464p^4r^4+(q^2-2gp^3)^2,\\
h_{2}&=3448p^4r^4-(q^2-2gp^3)^2,\\
h_{3}&=12p(2gp^3-q^2)r,\\
h_{4}&=24p^2qr^2.
\end{align*}

We show that $G = \gcd(h_1, h_2, h_3, h_4) = 1$. Suppose not, and let $s$ be a prime divisor of $G$. Then $s \mid h_1 + h_2 = 2^8 \cdot 3^3 \cdot p^4 r^4$.
If $s \mid p$, then, because $s \mid h_1$, we get $s \mid q^4$ — a contradiction. If $s \mid r$, then $s \mid q^2 - 2gp^3$ and $s \mid g(q^2 - gp^3)$. Thus,
$s \mid g(q^2 - gp^3) - g(q^2 - 2gp^3) = g^2 p^3.$ However, $\gcd(uv, w) = 1$, i.e., $\gcd(r, pqg^2) = 1$, and we get a contradiction. We have thus proved that each prime divisor of $G$ divides $ 6$.

If $s = 2$, then $2 \mid q$. Since $\gcd(p, q) = 1$ and $\gcd(pq, r) = 1$, we get $p \equiv r \equiv 1 \pmod{2}$. From the relation $gq^2 = g^2 p^3 + 62209 p r^4$,
we deduce $g^2 + r^4 \equiv 0 \pmod{4}$. Thus $g^2 + 1 \equiv 0 \pmod{4}$ — a contradiction.

If $s = 3$, then the condition $3 \mid h_1 + h_2$ implies $q^2 + gp^3 \equiv 0 \pmod{3}$. If $3 \mid q$, then $3 \mid p$ or $3 \mid g$. If $3 \mid p$, then we get a contradiction with the condition $\gcd(p, q) = 1$. On the other hand, if $3 \mid g$, then we get that $3 \mid pr^4$, and arrive at a contradiction with the condition $\gcd(pq, r) = 1$. We thus proved that $q \not\equiv 0 \pmod{3}$.

Now, from the congruence $h_4 \equiv 0 \pmod{3}$ and the condition $q^2 + gp^3 \equiv 0 \pmod{3}$, we get that $3 \mid p^4 r^4$, and thus $3 \mid p$ or $3 \mid r$. The case $3 \mid p$ was already eliminated. If $3 \mid r$, then the congruence $q^2 + gp^3 \equiv 0 \pmod{3}$, together with $gq^2 = g^2 p^3 + 62209 p r^4$, implies that $3 \mid g^2 p^3$. We get that $3 \mid g$, and arrive at the final contradiction $3 \mid q$.


Let us go back to the curve $E$. One can easily check that the torsion part of $E(\Q)$ has order 2 with the generator $T=(0,0)$. As computed with the help of Pari GP {\tt ellrank} procedure, the rank of $E$ is equal to 2 and the generators of the infinite part of $E(\Q)$ are given by
\begin{align*}
G_{1}&=\left(\frac{67600}{51^2}, \frac{169584220}{51^3}\right),\\
G_{2}&=\left(\frac{29686962975500601}{1803016^2}, \frac{5116943017644569380192365}{1803016^3}\right).
\end{align*}
Let us observe that if $(x, y)\in E(\Q)$, then necessarily $x\geq 0$. Moreover, the equation $x^3+62209x=0$ has only one real root $x=0$. Thus, the set $E(\R)$ of all real points on $E$ is connected. Invoking Hurwitz's \cite{Hu} theorem, we know that the set $E(\Q)$ is dense in the set $E(\R)$. In particular, if we write 
$$
P_{m, n, \varepsilon}=[m]G_{1}+[n]G_{2}+\varepsilon T=(x_{m,n,\varepsilon}, y_{m,n,\varepsilon}),
$$
then the set
$$
\cal{W}=\{x_{m,n,\varepsilon}:\;m, n\in\Z, \varepsilon\in\{0,1\}\}
$$
is dense in the set of positive real numbers. In particular, both sets $\cal{A}_{1}\cap \cal{W}, \cal{A}_{2}\cap \cal{W}$
are infinite. Thus, our construction guarantees, that for each triple $m, n, \varepsilon$ such that $f_{2}(x_{m, n, \varepsilon})>0$ ($f_{2}(x_{m, n, \varepsilon})<0$) we get a solution of the equation $x^3+y^3=a^4+b^4$ (of the equation $x^3-y^3=a^4+b^4$) in coprime positive integers.
\end{proof}

\begin{rem}
{\rm From our construction it follows that
$$
\lim_{N\rightarrow +\infty}C_{4}(\pm, +)(N)=+\infty.
$$
Unfortunately, we were unable to compute precise lower bound for this counting function. The problem is that the constructed integer solutions depends on the behaviour of rational points on associated elliptic curve. However, from the work of Silverman \cite{Sil} we know that the canonical height $\widehat{h}(P)$ of every non-torsion point $P$ on an elliptic curve $E/\Q$ satisfies $\hat{h}(P)>C_{1}\log (|\Delta(E)|)+C_{2}$. In particular, from this result one can deduce that $\#\{P\in E(\Q):\;h(P)<T\}\thicksim (\log T)^{r/2}$, where $r$ is the rank of $E$. Thus, it is reasonable to expect that $C_{4}(\pm, +)\gg (\log N)^{1/4}$.
}
\end{rem}

\begin{exam}
{\rm Let us put
\begin{align*}
\cal{P}_{+}=\{(m, n)\in\N\times\N:\;f(x_{m, n, 0})>0\},\\
\cal{P}_{-}=\{(m, n)\in\N\times\N:\;f(x_{m, n, 0})<0\}.
\end{align*}
We have
\begin{align*}
\cal{P}_{+}&=\{(0, 17), (1, 19), (3, 0), (4, 2), (5, 4), (6, 6), (7, 8), (8, 10), (9, 12), (10, 14), \ldots  \},\\
\cal{P}_{-}&=\{(0, 1), (1, 0), (1, 1), (1, 2), (2, 1), (0, 2), (0, 3), (0, 4), (2, 0), (4, 0), \ldots  \}.
\end{align*}

Let us take $(1, 0)\in\cal{P}_{-}$, i.e., we consider the point $G_{1}$. We thus have $u=67600, v=169584220, w=51$. We compute $g=\gcd(u, v)=260$ and then $p=260, q=652247, r=51$. Performing necessary calculations we get a coprime positive solution of $x^3-y^3=a^4+b^4$ in the form
\begin{equation*}
\begin{array}{ll}
  x=173401648246485776962081, & y=173187961382386571842081, \\
  a=66239528407912080, & b=2752392590812800.
\end{array}
\end{equation*}
On the other hand, if we take $(3, 0)\in\cal{P}_{+}$ and perform the same procedure we get another solution of $x^3+y^3=a^4+b^4$ in coprime positive integers, where $x, y$ have 209 and 207 digits respectively. This is the reason why we do not present this solution explicitly.
}
\end{exam}

\section{The equation $x^3\pm y^3=a^6\pm b^6$ in positive integers}\label{sec3}

Let us recall the following well known identity of Mahler \cite{Mah}
$$
(9x^4)^{3}+(3x-9x^4)^3=(9x^3-1)^3+1.
$$
By a simple computation, we note the following equalities of sets:
\begin{align*}
\cal{B}_{1}&=\{x\in\R:\;3x-9x^4>0 \wedge 9x^3-1>0\}=(3^{-2/3},3^{-1/3}),\\
\cal{B}_{2}&=\{x\in\R:\;3x-9x^4<0 \wedge 9x^3-1<0\}=(-\infty, 0),\\
\cal{B}_{3}&=\{x\in\R:\;3x-9x^4>0 \wedge 9x^3-1<0\}=(0,3^{-2/3})
\end{align*}
In particular, for each $i\in\{1, 2, 3\}$, the set $\cal{B}_{i}\cap \Q$ is infinite.

By taking $x=u/v$ and multiplying both sides by $v^{12}$ we get the identity
\begin{equation}\label{Mah}
(3u^2)^6+(3u(v^3-3u^3))^3=(v(9u^3-v^3))^3+v^{12}.
\end{equation}
Let us observe that that if $\gcd(u, v)=1$ and $\gcd(v,3)=1$, then
$$
\gcd(3u^2, 3u(v^3-3u^3), v(9u^3-v^3), v)=1.
$$
If $\gcd(u, v)=1$ and $\gcd(v,3)=3$, then we replace $v$ by $3v$ and by dividing by $3^{6}$ we get the identity
\begin{equation*}
(u^2)^6+(u(9v^3-u^3))^3=(3v(u^3-3v^3))^3+(3v^2)^{6},
\end{equation*}
where $\gcd(u^2, u(9v^3-u^3), 3v(u^3-3v^3), 3v^2)=1$.

As an immediate consequence of (\ref{Mah}) we get the following.

\begin{thm}\label{minussing}
For each choice of the sign the Diophantine equation $x^3\pm y^3=a^6-b^6$ has infinitely many nontrivial solutions in positive integers.
\end{thm}
\begin{proof}
To get positive integer solutions of the equation $x^3\pm y^3=a^6-b^6$, it is enough to rewrite (\ref{Mah}) as
$$
(3u(v^3-3u^3))^3-(v(9u^3-v^3))^3=v^{12}-(3u^2)^6.
$$
Indeed, in the case of the ``$-$'' sign, we note that for $u/v\in\cal{B}_{1}, \gcd(v, 3)=1$, we get coprime positive solutions $x=3u(v^3-3u^3), y=v(9u^3-v^3), a=v^{2}, b=3u^2$ in positive integers. In particular, for $t\in\N_{+}$ we can take $u=3t+1, v=6t+1$. We then have that $u/v\in\cal{B}_{1}$. Note that one can also take $u/v\in\cal{B}_{2}, \gcd(v,3)=1$ and switch the role of $x, y$. In consequence, the inequality $C_{6}(-,-)(N)\gg N^{1/12}$ is true.

On the other hand, in the case of the ``$+$'' sign, we just take $u/v\in\cal{B}_{3}, \gcd(v,3)=1$ and get coprime solutions $x=3u(v^3-3u^3), y=|v(9u^3-v^3)|, a=v^{2}, b=3u^2$ in positive integers. In particular, for $t\in\N_{+}$ we can take $u=t, v=3t+1$. We then have that $u/v\in\cal{B}_{3}$. Thus, the inequality $C_{6}(+,-)(N)\gg N^{1/12}$ is true.
\end{proof}


\begin{thm}\label{plussign}
For each choice of the signs the Diophantine equation $x^3\pm y^3=a^6+b^6$ has infinitely many nontrivial solutions in positive integers.
\end{thm}
\begin{proof}
We use the identity (\ref{Mah}) again. First we show that the Diophantine equation $w^2=3u(v^3-3u^3)$ has infinitely many solutions in coprime positive integers $(u, v, w)$. By substituting
$$
v=\frac{up}{3r^2},\quad w=\frac{u^2q}{3r^3},
$$
and dividing by $u^4$, we reduce the equation $w^2=3u(v^3-3u^3)=:g(u,v)$ to the equation $q^2=p^3-81r^6$. Equivalently, we are interested in rational points on the elliptic curve $E:\;y^2=x^3-81$ with $x=p/r^2, y=q/r^3$. One can easily check that the curve $E$ has trivial torsion. The rank of $E$ is equal to 1 and the generator of $E(\Q)$ is $P=(x_{1}, y_{1})=(13, 46)$. In particular the set $E(\Q)$ is infinite. Now, for a given $n\in\N_{+}$, let us write
$$
[n]P=\left(\frac{p_{n}}{r_{n}^2},\frac{q_{n}}{r_{n}^3}\right).
$$
It is clear that $p_{n}/r_{n}^2>3^{4/3}$ for $n\in\N_{+}$. We know that $\gcd(p_{n}q_{n}, r_{n})=1$. Moreover, due to the fact that the equation $x^3-81=0$ has only one real root, the set $E(\R)$ of all real points on $E$ is connected. Thus, invoking Hurwitz's theorem again, we know that the set $E(\Q)$ is dense in the set of $E(\R)$. In particular, the set
$$
\cal{V}=\{p_{n}/r_{n}^2:\;n\in\N\}
$$
is dense in the set $(3^{4/3},+\infty)$ in the Euclidean topology.

To get the solutions of $w^2=3u(v^3-3u^3)$ in coprime positive integers $u, v$, it is enough to take
\begin{equation}\label{plussol}
u_{n}=3r_{n}^2,\quad v_{n}=p_{n}, \quad w_{n}=3q_{n}r_{n},
\end{equation}
for $n\in\N_{+}$. The fact that $(u_{n}, v_{n}, w_{n})$ is a solution of our equation follows from the equality
$$
w_{n}^2-g(u_{n}, v_{n})=9r_{n}^2(q_{n}^2-(p_{n}^3-81r_{n}^6))=0.
$$

To finish the proof of our theorem, we observe that from the inequality $p_{n}/r_{n}^2>3^{4/3}$ we get the inequality $u_{n}/v_{n}=3r_{n}^2/p_{n}<3^{-1/3}$. From the density of the set $\cal{V}$ we get that there are infinitely many values of $n$ such that $u_{n}/v_{n}\in\cal{B}_{1}$. Thus, the integers
$$
x=v_{n}(9u_{n}^3-v_{n}^3),\quad y=v_{n}^{4},\quad a=3u_{n}^2,\quad b=w_{n}
$$
solve the equation $x^3+y^3=a^6+b^6$.

Using similar reasoning, we get that there are infinitely many values of $n$ such that $u_{n}/v_{n}\in\cal{B}_{3}$ and get infinitely many coprime positive integers
$$
x=v_{n}^{4},\quad y=|v_{n}(9u_{n}^3-v_{n}^3)|,\quad a=3u_{n}^2,\quad b=w_{n},
$$
which solve the equation $x^3-y^3=a^6+b^6$.

\end{proof}

\begin{rem}
{\rm As in the case of the equation $x^3\pm y^3=a^4+b^4$ it follows that
$$
\lim_{N\rightarrow +\infty}C_{6}(\pm, +)(N)=+\infty,
$$
but we were unable to compute the lower bound for this counting function. In any way, we expect that $C_{6}(\pm, +)(N)\gg (\log N)^{1/6}$. 
}
\end{rem}

\begin{exam}
{\rm From the proof of Theorem \ref{plussign} we know that the sets
\begin{align*}
\cal{P}_{+}=\{n\in\N_{+}:\;u_{n}/v_{n}\in\cal{B}_{1}\},\\
\cal{P}_{-}=\{n\in\N_{+}:\;u_{n}/v_{n}\in\cal{B}_{3}\}
\end{align*}
are infinite. We have
\begin{align*}
\cal{P}_{+}&=\{2, 6, 10, 11, 15, 19, 23, 27, 31, 32, 36, 40, 44, 48, 52, 53, 57, 61, 65, 69, 73, 74, \ldots\},\\
\cal{P}_{-}&=\{1, 3, 4, 5, 7, 8, 9, 12, 13, 14, 16, 17, 18, 20, 21, 22, 24, 25, 26, 28, 29, 30, 33,\ldots\}.
\end{align*}
If we take $n=1$ then $P=(13,46)$ and the corresponding solution of the equation $x^3-y^3=a^6+b^6$ is
$$
x=13^4,\;y=25402,\;a=27,\;b=138.
$$

On the other hand, if we take $n=2$ then $[2]P=(36985/8464, 1215181/778688)$ and the corresponding solution of the equation $x^3+y^3=a^6+b^6$ is
$$
x=3578403985453454495,\;y=36985^4,\;a=3^{3}\cdot 8464^{2},\;b=335389956.
$$

}
\end{exam}

\section{Some observations on general equation $x^3\pm y^3=a^k\pm b^k$}\label{sec4}

To gain better understanding of what we can expect in case of the Diophantine equation $x^3\pm y^3=a^k\pm b^k$ for $k\neq 4, 6$, we extended computations done by Wagstaff \cite{Wag}. Let us recall that in Wagstaff's work to find all positive solutions of the equation $x^3+y^3=a^k+b^k$ a method presented by Bernstein \cite{Ber}. The method is based on the comparison of the elements of two tables: one involving the sums $x^3+y^3$ and the second one containing the sums $a^k+b^k$. In the case of $k=5$ the range was $a\leq b\leq 5000$ (with appropriate bounds for $x, y$). In the case of $k=7$ the range was $a\leq b\leq 700$. Because we consider slightly more general equations involving signs we used a different method.

For a given $A\in\Z$ let us describe the method which can be used to find {\it all} integer solutions of the equation $x^3+y^3=A$. First of all, let us note that this is a very particular case of a Thue equation, i.e., an equation of the form $F(x, y)=A$, where $F\in\Z[x,y]$ is a binary form of degree $\geq 2$. In our case, the method of solution is based on the factorization of $A$. More precisely, one can compute the set $D(A)=\{d\in\N_{+}:\;d|A\;\mbox{and}\;d^2<A\}$. Then, for given $d\in D(A)$, we compute the solutions of the system
$$
x+y=d,\quad x^2-xy+y^2=\frac{A}{d}
$$
and get
$$
x=\frac{3d^2\pm \sqrt{3d(4A-d^3)}}{6d},\quad y=\frac{3d^2\mp \sqrt{3d(4A-d^3)}}{6d}.
$$
and check whether $x, y$ are integers. In case of small values of $A$ this method is very quick and allows us to find all integer solutions of the equation $x^3+y^3=A$. The drawback of this method is the necessity of finding the full factorization of $A$ and enumerating the set $D(A)$. On the other hand, in this way we can find all integer solutions of $x^3+y^3=A$ (to get this we also allow negative divisors of $A$, i.e., the values $-d$ with $d\in D(A)$). This is a useful property because during one computation we can get positive solutions of the equation $x^3+y^3=A$ as well as positive solutions of the equation $x^3-y^3=A$. This procedure can be implemented from scratch in any computer algebra system allowing factorizations of integers. Fortunately, a more general version of this approach, where instead of a particular form $x^3+y^3$ we can work with any {\it reducible} form, is implemented as a subroutine in the procedure {\tt thue()} in Pari GP. More precisely, for a given form $F\in\Z[x,y]$ and $A\in\Z$ the function {\tt thue} attempts to compute the set of integer solutions of the Thue equation $F(x,y)=A$. If $F$ is reducible, say $F(x,y)=F_{1}(x, y)F_{2}(x, y)$, the procedure works as follows \cite{Alo}. First, the resultant $R(u, v, y) = \op{Res}_{x}(F_{1}(x,y)-u,F_{2}(x,y)-v)$ is computed. Next for any $d\in D(A)$ we consider the polynomial $S(y)=R(d, A/d, y)$ and it is checked whether the equation $S(y)=0$ has an integer root. If $S(y_{0})=0$ for some $y_{0}\in\Z$, then the equation $F_{1}(x, y_{0})=0$ is checked for the integer roots. In this way all integer solutions of $F(x, y)=A$ can be computed. To initialize the search in Pari GP we need to use the procedure {\tt thueinit}. To see this in action we take $F(x,y)=x^3+y^3$ and $A=31213$:

{\tt
\noindent gp> T=thueinit(F(x,1)); \\
\noindent gp> thue(T,31213) \\
\noindent [[-35, 42], [21, 28], [28, 21], [42, -35]] \\
}
Thus, the only integer solutions of $x^3+y^3=31213$ are $(-35, 42), (21, 28)$ and its permutations.

\begin{rem}
{\rm Let us note, that in the case of a general (irreducible) form, to increase the speed of the computations the procedure {\tt thue} assumes validity of the Generalized Riemann Hypothesis. If we need to have an unconditional result, an additional flag during the initialization of {\tt thueinit} needs to be introduced. More precisely, instead of {\tt thueinit(F(x,1))} we need to write {\tt thueinit(F(x,1),1)}. Fortunately, as was explained above, in our case of the form $F(x,y)=x^3+y^3$ there is no need to do that. More information concerning the solver of Thue equations implemented in Pari GP can be found in the user's guide available through \cite{pari}.
}
\end{rem}

In our investigations of the Diophantine equation (\ref{generaleq}) we used a simple double loop for $a, b$ to enumerate all solutions of $x^3\pm y^3=a^k\pm b^k$. The computation was done on the Faculty of Mathematics and Computer Science of the Jagiellonian University cluster. The total time of computations was about three weeks. In case of $k=5$ and the choice of signs ``--+'' we found 99 solutions. In case of ``++'' signs we found 56 solutions. (The solutions sets are available upon reasonable request.) 

In the case of a more general equation $x^3\pm y^3=a^k-b^k$, with a relatively weak condition on $k$, we were able to spot some patterns and this allow us to prove the following.

\begin{thm}\label{3kidentities}
If $k\not\equiv 0\pmod{4}$, then for each choice of the sign, the Diophantine equation $x^3\pm y^3=a^k-b^k$ has infinitely many nontrivial solutions in positive integers. 
\end{thm}  
\begin{proof}
Let us note that it is enough to prove our statement for $k$ of the form $k=6(2m+1)$ for some $m\in\N$. Indeed, each integer non-divisible by 4 is a divisor of $6(2m+1)$ for suitable chosen $m$. The statement follows from the polynomials identities
$$
(3^{3m+2}x^{3(2m+1)}-1)^3+(3^{m+1}x^{2m+1}(3^{3m+1}x^{3(2m+1)}-1))^{3}=(3x^2)^{6(2m+1)}-1,
$$
and 
$$
(3^{m+1}x^{2m+1}(3^{3m+1}x^{3(2m+1)}+1))^{3}-(3^{3m+2}x^{3(2m+1)}+1)^3=(3x^2)^{6(2m+1)}-1.
$$
\end{proof}

From the above identity we immediately deduce the lower bound for $C_{k}(\pm, -)(N)$.
\begin{cor}
For $k\not\equiv 0\pmod{4}$ we have $C_{k}(\pm, -)(N)\gg N^{k/12(2m+1)}$, where $m$ is the smallest positive integer satisfying $6(2m+1)\equiv 0\pmod{k}$.
\end{cor}

As a byproduct of the above theorem we get the following.

\begin{cor}
If $k\not\equiv 0\pmod{4}$, then the system of Diophantine equations
$$
p^3-q^3=r^3+s^3=t^k-1
$$
has infinitely many solutions in positive integers.
\end{cor}


Motivated by our numerical search we dare to formulate the following
\begin{conj}\label{eq5}
For each choice of the signs the equation $x^3\pm y^3=a^5+b^5$ has infinitely many nontrivial solutions in positive integers.
\end{conj}

In this direction one can note the identity
$$
(8(9t^5+1))^3-(8(9t^5-1))^3=4^{5}+(12t^2)^5,
$$
which shows that the equation $x^3-y^3=a^5+b^5$ has infinitely many positive solutions satisfying the condition $\gcd(x, y, a, b)=4$.

\bigskip

The heuristic presented in Wagstaff's paper predicts that the Diophantine equation $x^3+y^3=a^7+b^7$  has only finitely many positive coprime solutions do not satisfying $\{x^{3}, y^{3}\}\neq \{a^{7}, b^{7}\}$. His computations confirms that there is no positive coprime solutions with $b\leq a\leq 700$. We dropped the assumption of positivity of solutions and performed numerical search of solutions of the equation $x^3\pm y^3=a^7+b^7$ in the range $0<b<a<50000$. In this range we found only one solution which lead to the identity
$$
1250534^3-637445^3=51^7+402^7.\\
$$

It is unclear whether Wagstaff's heuristic gives real prediction and the equation $x^3+y^3=a^7+b^7$. Our experience suggests the opposite and the identity
$$
(4u^7+108v^7)^3+(4u^7-108v^7)^3=(2u^3)^{7}+(6uv^2)^7,
$$
where $u, v\in\N$, $u$ is odd and $\gcd(u, v)=1$ and $u/v>3^{3/7}$, shows the existence of infinitely many positive solutions of the equation $x^3+y^3=a^7+b^7$ satisfying the condition $\gcd(x, y, a, b)=2$. In particular, this implies that the related Diophantine equation $x^3+y^3=2^7a^7+b^7$ has infinitely many solutions in coprime positive integers (but Wagstaff's heuristic suggest finiteness of the coprime solutions in positive integers).

In any way, in this context one can ask the following

\begin{ques}
For which values of $k\in\N_{\geq 3}$ and $d\in\N_{+}$, does the Diophantine equation $x^3+y^3=a^k+b^k$ have infinitely many solutions in, not necessarily positive, integers $(x, y, a, b)$ satisfying the condition $\gcd(x, y, a, b)=d$?
\end{ques}

\section{The case $k=4, 6$ again}\label{sec5}

In the light of our results concerning the existence of infinitely many solutions of $x^3\pm y^3=a^{k}+b^{k}$ in coprime positive integers for $k=4, 6$, it is natural to ask whether there is a solution in coprime polynomials. We were unable to find one but using your computational approach we calculated solutions in the range $0<b<a\leq 10^5$ in case $k=4$. In the case of ``$+$'' sign we found 355 solutions. In the case of ``$-$''sign we found 648 solutions. (The solutions sets are available upon reasonable request.) Because in the considered range the corresponding sets of solutions are relatively big, the expectation that there is a polynomial solution is reasonable. On the other hand, for $k=6$, in the range $0<b<a<50000$ we found exactly 28 coprime positive integer solutions of the equation with the ``$+$'' sign. In the same range we found 83 nontrivial positive integer solutions of the equation with the ``$-$'' sign. Although our numerical computations are not so convincing as in the case $k=4$, we formulate the following.

\begin{conj}
Let $k=4, 6$. For each choice of the sign there is a solution of the Diophantine equation $x^3\pm y^3=a^k+b^k$ in coprime polynomials $(x, y, a, b)$ with integer coefficients and positive leading terms such that no partial sum in the expression $x^3 \pm y^3-(a^k+b^k)$ is a zero polynomial.
\end{conj}

In this context it would be interesting to find infinitely many values of $d$ such that the elliptic curve
$$
E_{d}:\;y^2=x^3+(243d^8+1)x
$$
has positive rank. This curve appeared in the course of the proof of existence of infinite many $x\in\Q$ such that $g_{2}(x)$ is a square, where $g_{2}(x)$ appeared in (\ref{34identity}). Let us note that for each $d$ of the form $d=2^{k}3^{l}$ the reasoning used in the proof of Theorem \ref{equation34} guarantees the existence of infinitely many coprime positive solutions of the equation $x^3\pm y^3=a^4+b^4$ provided that $E_{d}$ has positive rank. Small numerical search shows that for the following pairs $(k, l)$ lead to the elliptic curve $E_{d}$ with positive rank:
$$
(k, l)=(0, 1), (0, 2), (0, 3), (0, 5), (1, 0), (4, 3), (5, 0),\ldots
$$
We expect that there are infinitely many such pairs but we were unable to prove such a statement. It should be easier to prove that there are infinitely many values of $d\in\N$ such that the curve $E_{d}$ has positive rank. However, note that if $d\neq 2^{k}3^{l}$, then some additional work need to be done to guarantee that the rational points on $E_{d}$ lead to coprime solutions of $x^3\pm y^3=a^4+b^4$.

\bigskip

We proved that for $k=4, 6$ and each choice of the signs, the Diophantine equation $x^3\pm y^3=a^k\pm b^k$ has infinitely many nontrivial solutions in positive integers. In the course of the proofs we obtained lower bounds for the quantities $C_{k}(\pm, -)(N)$ for $k=4, 6$ and proved that $C_{k}(\pm, +)(N)$ goes to the infinity with $N$. However, our bounds are very weak and we formulate the following.

\begin{ques}
Let $k\in\{4, 6\}$ or $k\not\equiv 0\pmod{4}$ and $N$ be a large positive number. What is the true order of magnitude of $C_{k}(\pm, \pm)(N)$ ?
\end{ques}

\bigskip

\noindent {\bf Acknowledgements.} The author is grateful to three anonymous referees for a careful reading of the paper and suggestions which improved the presentation.

\bigskip

\noindent Maciej Ulas, Jagiellonian University, Institute of Mathematics,
{\L}ojasiewicza 6, 30-348 Krak\'ow, Poland; email:\;{\tt maciej.ulas@uj.edu.pl}


\begin{thebibliography}{100}
\bibitem{Alo} B. Allombert, {\it Personal communication}.

\bibitem{Ber}  D. J. Bernstein, {\it Enumerating solutions to $p(a)+q(b)=r(c)+s(d)$}, Math. Comp. 70 (2000), 389--394.

\bibitem{Elk} N.D.~Elkies, On $A^4+B^4+C^4=D^4$, Math. Comp. 51 (1988), no. 184, 825--835.

\bibitem{Hu} A. Hurwitz, {\it \"{U}ber tern\"{a}re diophantische Gleichungen dritten Grades}, Vierteljahrschrift d. Naturforsch. Ges. \''{Z}urich 62, 207--229 (1917).

\bibitem{LP} L. J. Lander, T. R. Parkin, {\it Counterexample to Euler's conjecture on sums of like powers}, Bull. Amer. Math. Soc. 72(6) (1966), 1079--1079.

\bibitem{Lan} J. Lander, {\it Equal sums of unlike powers}, Fibonacci Quart. 28 (1990), 141--150.

\bibitem{Mah}  K. Mahler, {\it Note on Hypothesis K of Hardy and Littlewood}, J. London Math. Soc., 11 (2) (1936), 136--138.


\bibitem{pari} The PARI~Group, PARI/GP version \texttt{2.15.2}, Univ. Bordeaux, 2022, \url{http://pari.math.u-bordeaux.fr/}

\bibitem{Sil} J. H. Silverman, {\it Lower bound for the canonical height on elliptic curves}, Duke Math. J. 48 (1981), 633--648.

\bibitem{Wag} S. S. Wagstaff, Jr., {\it Equal sums of two distinct like powers}, J. Integer Sequences 25 (2022), Article 22.3.1.



\end{thebibliography}
\end{document}